\documentclass[11pt,a4paper]{amsart}
\usepackage{amsmath,amssymb, amsbsy}
\usepackage{subfigure}
\usepackage{graphpap,latexsym,epsf}
\usepackage{color,psfrag}
\usepackage[dvips]{graphicx}
\usepackage{enumerate}
\usepackage{bbm}
\usepackage{relsize}
\begin{document}
\newcommand{\dyle}{\displaystyle}
\newcommand{\R}{{\mathbb{R}}}
\newcommand{\Hi}{{\mathbb H}}
\newcommand{\Ss}{{\mathbb S}}
\newcommand{\N}{{\mathbb N}}
\newcommand{\Rn}{{\mathbb{R}^n}}
\newcommand{\F}{{\mathcal F}}
\newcommand{\ieq}{\begin{equation}}
\newcommand{\eeq}{\end{equation}}
\newcommand{\ieqa}{\begin{eqnarray}}
\newcommand{\eeqa}{\end{eqnarray}}
\newcommand{\ieqas}{\begin{eqnarray*}}
\newcommand{\eeqas}{\end{eqnarray*}}
\newcommand{\f}{\hat{f}}
\newcommand{\Bo}{\put(260,0){\rule{2mm}{2mm}}\\}
\newcommand{\1}{\mathlarger{\mathlarger{\mathbbm{1}}}}


\theoremstyle{plain}
\newtheorem{theorem}{Theorem} [section]
\newtheorem{corollary}[theorem]{Corollary}
\newtheorem{lemma}[theorem]{Lemma}
\newtheorem{proposition}[theorem]{Proposition}
\def\neweq#1{\begin{equation}\label{#1}}
\def\endeq{\end{equation}}
\def\eq#1{(\ref{#1})}


\theoremstyle{definition}
\newtheorem{definition}[theorem]{Definition}
\newtheorem{remark}[theorem]{Remark}
\numberwithin{figure}{section}

\title[Kinetic energy of the Langevin particle]{Kinetic energy of the Langevin particle}

\author[C. Escudero]{Carlos Escudero}
\address{}
\email{}


\date{\today}

\begin{abstract}
We compute the kinetic energy of the Langevin particle using different approaches. We build stochastic differential equations that
describe this physical quantity based on both the It\^o and Stratonovich stochastic integrals. It is shown that the It\^o equation
possesses a unique solution whereas the Stratonovich one possesses infinitely many, all but one absent of physical meaning. We discuss
how this fact matches with the existent discussion on the It\^o vs Stratonovich dilemma and the apparent preference towards the
Stratonovich interpretation in the physical literature.
\end{abstract}
\maketitle

\section{Introduction}

The position of the Langevin particle obeys the stochastic differential equation
\begin{eqnarray}\nonumber
m \, \frac{d^2 X_t}{dt^2} &=& -\gamma \, \frac{d X_t}{dt} + \sigma \, \xi_t, \\ \nonumber
\left. X_t \right|_{t=0} &=& X_0, \\ \nonumber
\left. \frac{d X_t}{dt} \right|_{t=0} &=& V_0,
\end{eqnarray}
where $\xi_t$ is Gaussian white noise and $m, \gamma, \sigma >0$.
Clearly, this is Newton second law for a particle subjected to both viscous damping and a random force.
It is a classical model for the random dispersal of a particle~\cite{langevin,uo}, that can be considered as a refined version
of Brownian Motion, and hence the alternative name \emph{Physical Brownian Motion}~\cite{fgl}.

Let $(\Omega,\mathcal{F},\{\mathcal{F}_t\}_{t \ge 0},\mathbb{P})$ be a filtered probability space
completed with the $\mathbb{P}-$null sets
in which a Wiener Process $\{W_t\}_{t \ge 0}$ is defined; moreover assume
$\mathcal{F}_t \supset \sigma(\{W_s, 0 \le s \le t\})$.
This equation can be written in the precise manner
\begin{eqnarray}\label{langevin}
m \, d V_t &=& -\gamma \, V_t \, dt + \sigma \, dW_t, \\ \nonumber
\frac{d X_t}{dt} &=& V_t, \\ \nonumber
\left. V_t \right|_{t=0} &=& V_0, \\ \nonumber
\left. X_t \right|_{t=0} &=& X_0,
\end{eqnarray}
where $X_0, V_0 \in L^2(\Omega)$ are $\mathcal{F}_0-$measurable random variables. Obviously, $V_t$ is the velocity of the Langevin particle.
It is straightforward to check that the classical theorem of existence and uniqueness of solution of stochastic differential
equations applies to this model~\cite{kuo,oksendal}.

If we formally take the limit $m \searrow 0$ we arrive at the model
\begin{eqnarray}\nonumber
\gamma \, \frac{d X_t}{dt} &=& \sigma \, \xi_t, \\ \nonumber
\left. X_t \right|_{t=0} &=& X_0,
\end{eqnarray}
which can be translated to the precise version
\begin{eqnarray}\nonumber
d X_t &=& \frac{\sigma}{\gamma} \, dW_t, \\ \nonumber
\left. X_t \right|_{t=0} &=& X_0;
\end{eqnarray}
its solution reads
\begin{equation}\nonumber
X_t= X_0 + \frac{\sigma}{\gamma}W_t.
\end{equation}
Clearly, its derivative is not well-defined, at least as a (function-valued) stochastic process and, moreover, for any $\Delta t >0$, we find
\begin{eqnarray}\nonumber
\frac{X_{t+\Delta t}-X_t}{\Delta t} &=& \frac{\sigma}{\gamma} \frac{W_{t+\Delta t}-W_t}{\Delta t} \Rightarrow \\ \nonumber
\mathbb{E}\left[ \left(\frac{X_{t+\Delta t}-X_t}{\Delta t}\right)^2 \right] &=& \frac{\sigma^2}{\gamma^2}
\frac{\mathbb{E}[(W_{t+\Delta t}-W_t)^2]}{\Delta t^2} \\ \nonumber
&=& \frac{\sigma^2}{\gamma^2}\frac{1}{\Delta t} \\ \nonumber
&\hspace{1cm} & \\ \nonumber
&\underset{\Delta t \searrow 0}{\longrightarrow}& \infty,
\end{eqnarray}
so the mean kinetic energy of the Brownian model is not well-defined.
In the next section we will show how this deficiency of the Brownian model can be solved with the Langevin model.

\section{Computation of the kinetic energy}
\label{kinetic}

The kinetic energy of the Langevin particle is
$$
K_t = \frac12 m V_t^2.
$$
To compute it, we can simple solve for $V_t$ to find
$$
V_t = e^{-(\gamma/m) t} V_0 + \frac{\sigma}{m} \, \int_0^t e^{(\gamma/m)(s-t)} \, dW_s;
$$
and therefore
\begin{equation}\label{kenergy}
K_t = \frac12 m \left[ e^{-(\gamma/m) t} V_0 + \frac{\sigma}{m} \, \int_0^t e^{(\gamma/m)(s-t)} \, dW_s \right]^2.
\end{equation}
We can compute its mean value
\begin{eqnarray}\nonumber
\mathbb{E}(K_t) &=& \frac12 m \, \mathbb{E} \left\{ \left[ e^{-(\gamma/m) t} V_0 + \frac{\sigma}{m} \, \int_0^t e^{(\gamma/m)(s-t)} \, dW_s \right]^2 \right\} \\ \nonumber
&=& \frac12 m \, \mathbb{E} \left( e^{-2(\gamma/m) t} V_0^2 \right) \\ \nonumber
&& + \sigma \, \mathbb{E} \left[ e^{-(\gamma/m) t} V_0 \, \int_0^t e^{(\gamma/m)(s-t)} \, dW_s \right] \\ \nonumber
&& + \frac12 \frac{\sigma^2}{m} \, \mathbb{E} \left\{ \left[ \int_0^t e^{(\gamma/m)(s-t)} \, dW_s \right]^2 \right\} \\ \nonumber
&=& \frac12 m \, e^{-2(\gamma/m) t} \, \mathbb{E} \left( V_0^2 \right) \\ \nonumber
&& + \sigma \, e^{-(\gamma/m) t} \, \mathbb{E} \left( V_0 \right) \, \mathbb{E} \left[ \int_0^t e^{(\gamma/m)(s-t)} \, dW_s \right] \\ \nonumber
&& + \frac12 \frac{\sigma^2}{m} \int_0^t e^{2(\gamma/m)(s-t)} \, ds \\ \nonumber
&=& \frac12 m \, e^{-2(\gamma/m) t} \mathbb{E} \left( V_0^2 \right) + \frac{\sigma^2}{4\gamma} \left[ 1 - e^{-2(\gamma/m) t} \right] \\ \nonumber
&\hspace{1cm} & \\ \nonumber
&\underset{t \nearrow \infty}{\longrightarrow}& \frac{\sigma^2}{4\gamma} >0,
\end{eqnarray}
where we have used, in the third equality, the independence of $V_0$ and the Wiener integral (by assumption on $V_0$), and the It\^o
isometry, while the zero mean property of the Wiener integral has been used in the fourth.

Also note that
\begin{equation}\nonumber
\lim_{m \searrow 0} \mathbb{E}(K_t) = \frac{\sigma^2}{4\gamma},
\end{equation}
so the vanishing mass limit of the Langevin model allows to define a value of the mean kinetic energy for the Brownian model.
Moreover, this value coincides with the long time limit of the mean kinetic energy of the Langevin particle.
This coincidence matches well with previous analyses that showed the agreement of the long time limit of the Langevin model
and the Brownian one with respect to the dispersal of the trajectories~\cite{gillespie}.

\section{Direct computation}
\label{direct}

One could instead directly compute the stochastic differential equation
obeyed by the stochastic process $K_t$.
But this equation cannot be interpreted samplewise as~\eqref{langevin}
and requires to incorporate a notion of stochastic integration (different from the Wiener one);
herein we consider the stochastic integrals of It\^o~\cite{ito1,ito2} and Stratonovich~\cite{stratonovich}
as has been traditionally done in the
physical literature, and what coincides with the historical development.
From now on let us assume $V_0 \in L^4(\Omega)$ is a $\mathcal{F}_0-$measurable random variable.
Using It\^o calculus and following~\cite{kampen}
one arrives at
\begin{equation}\label{kito}
d K_t = \frac{\sigma^2}{2m} dt -2\frac{\gamma}{m} \, K_t \, dt + \sqrt{2 \frac{\sigma^2}{m} \, K_t} \, dW_t,
\qquad \left. K_t \right|_{t=0}= \frac12 m V_0^2;
\end{equation}
alternatively, using Stratonovich calculus and following~\cite{kampen} one gets
\begin{equation}\label{kstrat}
d K_t = -2\frac{\gamma}{m} \, K_t \, dt + \sqrt{2 \frac{\sigma^2}{m} \, K_t} \circ dW_t, \qquad \left. K_t \right|_{t=0}= \frac12 m V_0^2.
\end{equation}

\begin{remark}
There is a well known formula that connects It\^o and Stratonovich stochastic differential equations by means of a drift redefinition~\cite{kuo,oksendal}.
If one {\it formally} applied this formula to the present situation one would be tempted to conclude that equations~\eqref{kito} and~\eqref{kstrat}
are equivalent. However, this formula requires the continuous differentiability of the diffusion of the given stochastic differential equations,
a requirement that is not fulfilled by the square root diffusions of the present models. The lack of validity of this formula in such situations was
demonstrated in~\cite{ce}, and in the following it will become apparent again that it may lead to the creation or destruction of solutions when
applied out of its range of validity.
\end{remark}

We note that equation~\eqref{kito} possesses a unique solution, which is both strong and global, by the Wanatabe-Yamada theorem~\cite{wy};
however, this theorem is not applicable to equation~\eqref{kstrat} or, in general, to stochastic differential equations in Stratonovich form~\cite{ce}.
It is a simple exercise of stochastic calculus to check that formula~\eqref{kenergy} solves both
equations~\eqref{kito} and~\eqref{kstrat}. On the other hand, consider equation~\eqref{kstrat} subject to the initial condition
\begin{equation}\label{kstrat0}
\left. K_t \right|_{t=0}=0;
\end{equation}
it is clear that $K_t=0$ is an absorbing state for this equation and, given this initial condition, it is a global solution to it too.
Nevertheless, it is not an absorbing state for equation~\eqref{kito}, which we remind possesses a unique solution. Clearly, the stochastic
differential equation~\eqref{kstrat} subject to the initial condition~\eqref{kstrat0} possesses at least two solutions: $K_t=0$ and~\eqref{kenergy}. Actually, it is easy to combine both
to get the family of solutions
\begin{equation}\nonumber
K_t = \frac{\sigma^2}{2m} \left[ \int_{\lambda}^{t} e^{(\gamma/m)(s-t)} \, dW_s \right]^2 \,
\mathlarger{\mathlarger{\mathbbm{1}}}_{t > \lambda},
\end{equation}
where $\lambda \ge 0$ is an arbitrary parameter; that is, equation~\eqref{kstrat} subject to~\eqref{kstrat0} admits an uncountable number of solutions.
This fact, apparently, remained unseen before~\cite{kampen,west}.

Now we can focus again on the general case~\eqref{kstrat}. If we choose an $\omega \in \Omega$ such that $V_0(\omega)=0$ then the
problem reduces to the previous case; so we consider instead those samples $\omega \in \Omega$ such that $V_0(\omega) \neq 0$ and
consequently $V_0(\omega)^2 >0$. For such an $\omega$ the equation~\eqref{kstrat} possesses a unique solution up to some
stopping time $\mathcal{T}(\omega)$ that is positive almost surely; for such a time interval the solution is given by~\eqref{kenergy}.
Given that this equation falls under the assumptions of the classical existence and uniqueness theorem~\cite{kuo,oksendal} while $K_t>0$, we conclude
$$
\mathcal{T}(\omega) = \inf \{ t>0 : K_t=0 \} =: T_1(\omega).
$$
Now we can construct at least the following family of solutions to~\eqref{kstrat}:
\begin{eqnarray}\nonumber
K_t &=& \frac12 m \left[ e^{-(\gamma/m) t} V_0 + \frac{\sigma}{m} \, \int_0^t e^{(\gamma/m)(s-t)} \, dW_s \right]^2
\, \mathlarger{\mathlarger{\mathbbm{1}}}_{t < T_1(\omega)} \\ \nonumber
& & + \frac{\sigma^2}{2m} \left[ \int_{\lambda+T_1(\omega)}^{t} e^{(\gamma/m)(s-t)} \, dW_s \right]^2 \,
\mathlarger{\mathlarger{\mathbbm{1}}}_{t > \lambda + T_1(\omega)},
\end{eqnarray}
where $\lambda \ge 0$ is arbitrary.

Additionally define recursively the family of stopping times
$$
T_n := \inf \{ t>T_{n-1} + \lambda_{n-1} + \tau_{n-1} : K_t=0 \}, \quad n=2,3,\cdots,
$$
where $\{\tau_n\}_{n=1}^\infty$ is an arbitrary sequence of almost surely positive,
$L^0(\Omega)$, and $\mathcal{F}_{T_{n}(\omega)}-$measurable random variables.
This allows to extend the family of solutions to
\begin{eqnarray}\nonumber
K_t &=& \frac12 m \left[ e^{-(\gamma/m) t} V_0 + \frac{\sigma}{m} \, \int_0^t e^{(\gamma/m)(s-t)} \, dW_s \right]^2
\, \mathlarger{\mathlarger{\mathbbm{1}}}_{t < T_1(\omega)} \\ \nonumber & & +
\frac{\sigma^2}{2m} \left[ \int_{\lambda_1+T_1(\omega)}^{t} e^{(\gamma/m)(s-t)} \, dW_s \right]^2 \,
\mathlarger{\mathlarger{\mathbbm{1}}}_{\lambda_1 + T_1(\omega)<t<T_2(\omega)}
\\ \nonumber & & +
\frac{\sigma^2}{2m} \left[ \int_{\lambda_2+T_2(\omega)}^{t} e^{(\gamma/m)(s-t)} \, dW_s \right]^2 \,
\mathlarger{\mathlarger{\mathbbm{1}}}_{\lambda_2 + T_2(\omega)<t<T_3(\omega)}
\\ \nonumber & & + \cdots
\\ \nonumber & & +
\frac{\sigma^2}{2m} \left[ \int_{\lambda_n+T_n(\omega)}^{t} e^{(\gamma/m)(s-t)} \, dW_s \right]^2 \,
\mathlarger{\mathlarger{\mathbbm{1}}}_{\lambda_n + T_n(\omega)<t<T_{n+1}(\omega)}
\\ \nonumber & & + \cdots,
\end{eqnarray}
where $\{\lambda_n\}_{n=1}^\infty$ is an arbitrary sequence of non-negative real numbers.

Finally, we can build yet another extension of our set of solutions to
\begin{eqnarray}\nonumber
K_t &=& \frac12 m \left[ e^{-(\gamma/m) t} V_0 + \frac{\sigma}{m} \, \int_0^t e^{(\gamma/m)(s-t)} \, dW_s \right]^2
\, \mathlarger{\mathlarger{\mathbbm{1}}}_{t < T_1(\omega)}
\\ \nonumber & & +
\sum_{n=1}^{\infty} \frac{\sigma^2}{2m} \left[ \int_{\lambda_n+T_n(\omega)}^{t} e^{(\gamma/m)(s-t)} \, dW_s \right]^2 \,
\mathlarger{\mathlarger{\mathbbm{1}}}_{\lambda_n + T_n(\omega)<t<T_{n+1}(\omega)},
\end{eqnarray}
if $V_0(\omega) \neq 0$, for any sequence $\{\lambda_n\}_{n=1}^\infty$ of almost surely non-negative,
$L^0(\Omega)$, and $\mathcal{F}_{T_n(\omega)}-$measurable random variables, and
\begin{eqnarray}\label{kt1}
K_t &=& \frac12 m \left[ e^{-(\gamma/m) t} V_0 + \frac{\sigma}{m} \, \int_0^t e^{(\gamma/m)(s-t)} \, dW_s \right]^2
\, \mathlarger{\mathlarger{\mathbbm{1}}}_{t < T_1'(\omega)}
\\ \nonumber & & +
\sum_{n=1}^{N} \frac{\sigma^2}{2m} \left[ \int_{\mu_n+T_n'(\omega)}^{t} e^{(\gamma/m)(s-t)} \, dW_s \right]^2 \,
\mathlarger{\mathlarger{\mathbbm{1}}}_{\mu_n + T_n'(\omega)<t<T_{n+1}'(\omega)},
\end{eqnarray}
if $V_0(\omega) \neq 0$, for any finite sequence $\{\mu_n\}_{n=1}^N$, $N=1,2,\cdots$, of almost surely non-negative,
$L^0(\Omega)$, and $\mathcal{F}_{T_n'(\omega)}-$measurable
random variables, where
$$
T_1' :=T_1, \quad T_n' := \inf \{ t>T_{n-1}' + \mu_{n-1} + \tau_{n-1}' : K_t=0 \}, \quad n=2,3,\cdots, N+1,
$$
where $\{\tau_n'\}_{n=1}^N$ is an arbitrary finite sequence of almost surely positive,
$L^0(\Omega)$, and $\mathcal{F}_{T_{n}'(\omega)}-$measurable random variables.
The solution
\begin{equation}\label{kt2}
K_t = \frac12 m \left[ e^{-(\gamma/m) t} V_0 + \frac{\sigma}{m} \, \int_0^t e^{(\gamma/m)(s-t)} \, dW_s \right]^2
\, \mathlarger{\mathlarger{\mathbbm{1}}}_{t < T_1(\omega)},
\end{equation}
is also acceptable again if $V_0(\omega) \neq 0$. If $V_0(\omega) = 0$ we have the solutions:
\begin{eqnarray}\nonumber
K_t &=& \frac{\sigma^2}{2m} \left[ \int_{\bar{\lambda}_0}^{t} e^{(\gamma/m)(s-t)} \, dW_s \right]^2 \,
\mathlarger{\mathlarger{\mathbbm{1}}}_{\bar{\lambda}_0 <t<S_{1}(\omega)} \\ \nonumber & & +
\sum_{n=1}^{\infty} \frac{\sigma^2}{2m} \left[ \int_{\bar{\lambda}_n+S_n(\omega)}^{t} e^{(\gamma/m)(s-t)} \, dW_s \right]^2 \,
\mathlarger{\mathlarger{\mathbbm{1}}}_{\bar{\lambda}_n + S_n(\omega)<t<S_{n+1}(\omega)},
\end{eqnarray}
and
\begin{eqnarray}\label{kt3}
K_t &=& \frac{\sigma^2}{2m} \left[ \int_{\bar{\mu}_0}^{t} e^{(\gamma/m)(s-t)} \, dW_s \right]^2 \,
\mathlarger{\mathlarger{\mathbbm{1}}}_{\bar{\mu}_0 <t<S_{1}'(\omega)} \\ \nonumber & & +
\sum_{n=1}^{N} \frac{\sigma^2}{2m} \left[ \int_{\bar{\mu}_n+S_n'(\omega)}^{t} e^{(\gamma/m)(s-t)} \, dW_s \right]^2 \,
\mathlarger{\mathlarger{\mathbbm{1}}}_{\bar{\mu}_n + S_n'(\omega)<t<S_{n+1}'(\omega)},
\end{eqnarray}
for any sequences (finite the second) $\{\bar{\lambda}_n\}_{n=0}^\infty$ and $\{\bar{\mu}_n\}_{n=0}^N$ ($N=1,2,\cdots$)
of non-negative, $L^0(\Omega)$, and respectively $\mathcal{F}_{S_n(\omega)}-$measurable and $\mathcal{F}_{S_n'(\omega)}-$measurable
random variables, and where
\begin{eqnarray}\nonumber
S_n &:=& \inf \{ t>S_{n-1} + \bar{\lambda}_{n-1} + \bar{\tau}_{n-1} : K_t=0 \}, \quad n=1,2,\cdots, \\ \nonumber
S_n' &:=& \inf \{ t>S_{n-1}' + \bar{\mu}_{n-1} + \bar{\tau}_{n-1}' : K_t=0 \}, \quad n=1,2,\cdots, N+1,
\end{eqnarray}
with $S_0 = \bar{\lambda}_{-1} + \bar{\tau}_{-1}$ and $S_0' = \bar{\mu}_{-1} + \bar{\tau}_{-1}'$,
where $\{\bar{\tau}_n\}_{n=0}^\infty$ and $\{\bar{\tau}_n'\}_{n=0}^N$ are arbitrary sequences (finite sequence the latter) of almost surely positive,
$L^0(\Omega)$, and respectively $\mathcal{F}_{S_{n}(\omega)}-$measurable and $\mathcal{F}_{S_{n}'(\omega)}-$measurable random variables;
also, $\bar{\lambda}_{-1}$ and $\bar{\mu}_{-1}$ are two arbitrary almost surely non-negative,
$L^0(\Omega)$, and $\mathcal{F}_{0}-$measurable random variables,
and $\bar{\tau}_{-1}$ and $\bar{\tau}_{-1}'$ are almost surely positive,
$L^0(\Omega)$, and respectively $\mathcal{F}_{\bar{\lambda}_{-1}}-$measurable and $\mathcal{F}_{\bar{\mu}_{-1}}-$measurable random variables.
Yet the solution $K_t \equiv 0$ is also acceptable.

\section{Time scale of the spurious solutions}

In this section we show that the appearance of the spurious solutions has a well-defined time scale.
Since we are discussing the mathematical properties of a physical model, it is important to establish
the observability of these solutions.

\begin{theorem}\label{thmfpt}
The mean first passage time to zero as a function of the initial kinetic energy $K$ is given by the formula:
\begin{equation}\nonumber
\mathfrak{T}(K) = \frac{m\sqrt{\pi}}{\gamma} \, \exp\left(\frac{2 \gamma}{\sigma^2} K \right) D_{+}\left(\frac{\sqrt{2 \gamma}}{\sigma} K^{1/2} \right)
-  \frac{2m}{\gamma} \int_0^{\frac{\sqrt{2 \gamma}}{\sigma}K^{1/2}} D_{-}(w)\, dw,
\end{equation}
where the Dawson integrals
\begin{eqnarray}\nonumber
D_{+}(\cdot) &:=& e^{-(\cdot)^2} \int_{0}^{(\cdot)} e^{u^2}du, \\ \nonumber
D_{-}(\cdot) &:=& e^{(\cdot)^2} \int_{0}^{(\cdot)} e^{-u^2}du.
\end{eqnarray}
\end{theorem}

\begin{proof}
Since
\begin{equation}\label{mv2}
K_t=\frac12 m V_t^2,
\end{equation}
it is clear that $K_t=0 \Leftrightarrow V_t=0$, and $K_t>0 \Leftrightarrow V_t \neq 0$. Then the kinetic energy and the velocity become
zero simultaneously, and therefore it suffices to study the stopping time
$$
\mathfrak{T} := \inf \{ t>0 : V_t=0 \}.
$$
Since $V_t$ obeys the Ornstein-Uhlenbeck stochastic differential equation
$$
m \, d V_t = -\gamma \, V_t \, dt + \sigma \, dW_t,
$$
then $\mathfrak{T}_M=\mathfrak{T}_M(v)$ solves
$$
\frac{\sigma^2}{2m} \, \frac{d^2 \mathfrak{T}_M}{dv^2} - \gamma v \frac{d \mathfrak{T}_M}{dv} = -m,
$$
subjected to $\mathfrak{T}_M(0)=\partial_v\mathfrak{T}_M(M)=0$ and either $v \in [0,M]$ or $v \in [M,0]$ depending on the sign of $M$, with $M \neq 0$.
The solution reads
$$
\mathfrak{T}_M(v)= \frac{2m}{\gamma} \int_0^{\frac{\sqrt{m \gamma}}{\sigma}|v|}
\left[\exp\left(w^2- \frac{m \gamma}{\sigma^2}M^2\right) D_{-}\left(\frac{\sqrt{m \gamma}}{\sigma}|M|\right) - D_{-}(w)\right] dw.
$$
Now
\begin{eqnarray}\nonumber
\mathfrak{T}(v) &=& \lim_{M \to \pm \infty} \mathfrak{T}_M(v) \\ \nonumber
&=& \lim_{M \to \pm \infty} \frac{2m}{\gamma} \int_0^{\frac{\sqrt{m \gamma}}{\sigma}|v|}
\left[\exp\left(w^2- \frac{m \gamma}{\sigma^2}M^2\right) D_{-}\left(\frac{\sqrt{m \gamma}}{\sigma}|M|\right) - D_{-}(w)\right] dw \\ \nonumber
&=& \lim_{M \to \pm \infty} \frac{m}{\gamma} \Bigg\{ \sqrt{\pi} \left[2 \Phi\left(\frac{\sqrt{2m \gamma}}{\sigma}|M|\right)-1 \right]
\exp\left(\frac{m \gamma}{\sigma^2}v^2\right) D_{+}\left(\frac{\sqrt{m \gamma}}{\sigma}|v|\right) \\ \nonumber
&& \hspace{1.94cm} -\frac{m \gamma}{\sigma^2}v^2 \, {_2F_2}\left(1,1;3/2,2;\frac{m \gamma}{\sigma^2}v^2\right)\Bigg\} \\ \nonumber
&=&\sup_{M \in \mathbb{R}} \frac{m}{\gamma} \Bigg\{ \sqrt{\pi} \left[2 \Phi\left(\frac{\sqrt{2m \gamma}}{\sigma}|M|\right)-1 \right]
\exp\left(\frac{m \gamma}{\sigma^2}v^2\right) D_{+}\left(\frac{\sqrt{m \gamma}}{\sigma}|v|\right) \\ \nonumber
&& \hspace{1.51cm} -\frac{m \gamma}{\sigma^2}v^2 \, {_2F_2}\left(1,1;3/2,2;\frac{m \gamma}{\sigma^2}v^2\right) \Bigg\} \\ \nonumber
&=& \frac{m\sqrt{\pi}}{\gamma} \, \exp\left(\frac{m \gamma}{\sigma^2}v^2\right) D_{+}\left(\frac{\sqrt{m \gamma}}{\sigma}|v|\right)
-\frac{m^2}{\sigma^2} v^2 \, {_2F_2}\left(1,1;3/2,2;\frac{m \gamma}{\sigma^2}v^2\right) \\ \nonumber
&=& \frac{m\sqrt{\pi}}{\gamma} \, \exp\left(\frac{m \gamma}{\sigma^2}v^2\right) D_{+}\left(\frac{\sqrt{m \gamma}}{\sigma}|v|\right)
-  \frac{2m}{\gamma} \int_0^{\frac{\sqrt{m \gamma}}{\sigma}|v|} D_{-}(w)\, dw \\ \nonumber
&<& \infty
\end{eqnarray}
for all $v \in \mathbb{R}$, where
\begin{equation}\nonumber
\Phi(\cdot) := \text{Prob}(X \le \cdot) \quad \text{with} \quad X \sim \mathcal{N}(0,1),
\end{equation}
and ${_2F_2}(\cdot,\cdot;\cdot,\cdot;\cdot)$ is a generalized hypergeometric function~\cite{nist}.
The statement follows from changing variables as in~\eqref{mv2}.
\end{proof}

\begin{corollary}\label{zeroas}
Let the initial kinetic energy of a Langevin particle be positive. Then it becomes zero in finite mean time and,
in particular, it becomes zero in finite time almost surely.
\end{corollary}

\begin{proof}
This corollary is directly implied by the proof of Theorem~\ref{thmfpt}.
\end{proof}

The explicit formula in the statement of Theorem~\ref{thmfpt}
describes the time scale of observability of the spurious solutions as a function of the initial kinetic energy.
We note that this quantity is not just always finite, as stated in Corollary~\ref{zeroas},
but it can also be arbitrarily small (depending on the initial condition),
as an initial kinetic energy $K=0$ leads to an immediate observation of them.
Therefore the unphysical properties of the spurious solutions in section~\ref{direct} can be observed in a well defined
time scale, which can be made arbitrarily short by tuning the initial condition.

\section{Long time behavior of a class of spurious solutions}
\label{ltb}

In this section we show that the spurious solutions we have constructed are indeed spurious as they do not obey important physical laws.
For technical reasons we limit our analysis to a subclass of the explicit solutions built in section~\ref{direct}.

\begin{theorem}\label{thltb}
Let $K_t$ be as in~\eqref{kt1}, \eqref{kt2}, \eqref{kt3}, or $K_t \equiv 0$;
moreover assume $\{\tau_n'\}_{n=1}^N$, $\{\mu_n\}_{n=1}^{N}$, $\{\bar{\tau}_n'\}_{n=-1}^N$, and $\{\bar{\mu}_n\}_{n=-1}^{N}$
are finite almost surely. Then
$$
\mathbb{E}(K_t) \underset{t \nearrow \infty}{\longrightarrow} 0.
$$
\end{theorem}

\begin{proof}
The case $K_t \equiv 0$ is obvious. First note that
$$
K_t \underset{t \nearrow \infty}{\longrightarrow} 0 \qquad \text{almost surely}
$$
in all three cases~\eqref{kt1}, \eqref{kt2}, and \eqref{kt3} as a consequence of the assumptions on
$\{\tau_n'\}_{n=1}^N$, $\{\mu_n\}_{n=1}^{N}$, $\{\bar{\tau}_n'\}_{n=-1}^N$, and $\{\bar{\mu}_n\}_{n=-1}^{N}$,
and Corollary~\ref{zeroas}.

We start proving that all of these families of solutions are bounded in $L^1(\Omega)$ uniformly in $t$.
Let us begin with~\eqref{kt1} by noting
\begin{eqnarray}\label{young}
|K_t| &=& \Bigg| \frac12 m \left[ e^{-(\gamma/m) t} V_0 + \frac{\sigma}{m} \, \int_0^t e^{(\gamma/m)(s-t)} \, dW_s \right]^2
\, \mathlarger{\mathlarger{\mathbbm{1}}}_{t < T_1'(\omega)}
\\ \nonumber & & +
\sum_{n=1}^{N} \frac{\sigma^2}{2m} \left[ \int_{\mu_n+T_n'(\omega)}^{t} e^{(\gamma/m)(s-t)} \, dW_s \right]^2 \,
\mathlarger{\mathlarger{\mathbbm{1}}}_{\mu_n + T_n'(\omega)<t<T_{n+1}'(\omega)} \Bigg|
\\ \nonumber
&\le& m e^{-2(\gamma/m) t} V_0^2 + \frac{\sigma^2}{m} \left[ \int_0^t e^{(\gamma/m)(s-t)} \, dW_s \right]^2
\\ \nonumber & & +
\sum_{n=1}^{N} \frac{\sigma^2}{2m} \left[ \int_{\mu_n+T_n'(\omega)}^{t} e^{(\gamma/m)(s-t)} \, dW_s \right]^2 \,
\mathlarger{\mathlarger{\mathbbm{1}}}_{\mu_n + T_n'(\omega)<t},
\end{eqnarray}
by Young inequality.
Then in consequence
\begin{eqnarray}\nonumber
\mathbb{E}(|K_t|) &\le& m e^{-2(\gamma/m) t} \mathbb{E}(V_0^2)
+ \frac{\sigma^2}{m} \mathbb{E} \left\{ \left[ \int_0^t e^{(\gamma/m)(s-t)} \, dW_s \right]^2 \right\}
\\ \nonumber & & +
\sum_{n=1}^{N} \frac{\sigma^2}{2m} \mathbb{E} \left\{ \left[ \int_{\mu_n+T_n'(\omega)}^{t} e^{(\gamma/m)(s-t)} \, dW_s \right]^2 \,
\mathlarger{\mathlarger{\mathbbm{1}}}_{\mu_n + T_n'(\omega)<t} \right\} \\ \nonumber
&\le& m \mathbb{E}(V_0^2)
+ \frac{\sigma^2}{m} \int_0^t e^{2(\gamma/m)(s-t)} \, ds
\\ \nonumber & & +
\sum_{n=1}^{N} \frac{\sigma^2}{2m} \mathbb{E} \left\{ \mathbb{E} \left\{
\left[ \int_{\mu_n+T_n'(\omega)}^{t} e^{(\gamma/m)(s-t)} \, dW_s \right]^2 \,
\mathlarger{\mathlarger{\mathbbm{1}}}_{\mu_n + T_n'(\omega)<t}
\Bigg| \mathcal{F}_{\mu_n + T_n'(\omega)} \right\} \right\} \\ \nonumber
&=& m \mathbb{E}(V_0^2)
+ \frac{\sigma^2}{2 \gamma} \left[1-e^{-2(\gamma/m)t}\right]
\\ \nonumber & & +
\sum_{n=1}^{N} \frac{\sigma^2}{2m} \mathbb{E} \left\{
\int_{\mu_n+T_n'(\omega)}^{t} e^{2(\gamma/m)(s-t)} \, ds \,
\mathlarger{\mathlarger{\mathbbm{1}}}_{\mu_n + T_n'(\omega)<t}
\right\} \\ \nonumber
&\le& m \mathbb{E}(V_0^2)
+ \frac{\sigma^2}{2 \gamma}
\\ \nonumber & & +
\sum_{n=1}^{N} \frac{\sigma^2}{4 \gamma} \mathbb{E} \left\{
\left[1-e^{2(\gamma/m)(\mu_n + T_n'(\omega)-t)}\right] \,
\mathlarger{\mathlarger{\mathbbm{1}}}_{\mu_n + T_n'(\omega)<t}
\right\} \\ \nonumber
&\le& m \mathbb{E}(V_0^2)
+ \frac{\sigma^2}{2 \gamma} \left( 1 + \frac{N}{2} \right) \\ \nonumber
&<& \infty,
\end{eqnarray}
by the It\^o isometry and the tower property of conditional expectation.
Arguing analogously for~\eqref{kt2} we find
\begin{eqnarray}\nonumber
\mathbb{E}(|K_t|) &=& \mathbb{E} \left\{ \frac12 m \left[ e^{-(\gamma/m) t} V_0 + \frac{\sigma}{m} \, \int_0^t e^{(\gamma/m)(s-t)} \, dW_s \right]^2
\, \mathlarger{\mathlarger{\mathbbm{1}}}_{t < T_1(\omega)} \right\} \\ \nonumber
&\le& \mathbb{E} \left\{ m e^{-2(\gamma/m) t} V_0^2 + \frac{\sigma^2}{m} \left[ \int_0^t e^{(\gamma/m)(s-t)} \, dW_s \right]^2
\right\} \\ \nonumber
&=& m e^{-2(\gamma/m) t} \mathbb{E}(V_0^2) + \frac{\sigma^2}{m} \int_0^t e^{2(\gamma/m)(s-t)} \, ds \\ \nonumber
&\le& m \mathbb{E}(V_0^2) + \frac{\sigma^2}{2 \gamma} \left[ 1 - e^{-2(\gamma/m)t} \right] \\ \nonumber
&\le& m \mathbb{E}(V_0^2) + \frac{\sigma^2}{2 \gamma} \\ \nonumber
&<& \infty.
\end{eqnarray}
And for~\eqref{kt3} we get
\begin{eqnarray}\nonumber
\mathbb{E}(|K_t|) &=& \frac{\sigma^2}{2m} \mathbb{E} \left\{ \left[ \int_{\bar{\mu}_0}^{t} e^{(\gamma/m)(s-t)} \, dW_s \right]^2 \,
\mathlarger{\mathlarger{\mathbbm{1}}}_{\bar{\mu}_0 < t <S_{1}'(\omega)} \right\} \\ \nonumber & & +
\sum_{n=1}^{N} \frac{\sigma^2}{2m} \mathbb{E} \left\{ \left[ \int_{\bar{\mu}_n+S_n'(\omega)}^{t} e^{(\gamma/m)(s-t)} \, dW_s \right]^2 \,
\mathlarger{\mathlarger{\mathbbm{1}}}_{\bar{\mu}_n + S_n'(\omega)<t<S_{n+1}'(\omega)} \right\} \\ \nonumber
&\le& \frac{\sigma^2}{2m} \mathbb{E} \left\{ \left[ \int_{\bar{\mu}_0}^{t} e^{(\gamma/m)(s-t)} \, dW_s \right]^2 \,
\mathlarger{\mathlarger{\mathbbm{1}}}_{\bar{\mu}_0 < t} \right\} \\ \nonumber & & +
\sum_{n=1}^{N} \frac{\sigma^2}{2m} \mathbb{E} \left\{ \left[ \int_{\bar{\mu}_n+S_n'(\omega)}^{t} e^{(\gamma/m)(s-t)} \, dW_s \right]^2 \,
\mathlarger{\mathlarger{\mathbbm{1}}}_{\bar{\mu}_n + S_n'(\omega)<t} \right\} \\ \nonumber
&=& \frac{\sigma^2}{2m} \mathbb{E} \left\{ \mathbb{E} \left\{ \left[ \int_{\bar{\mu}_0}^{t} e^{(\gamma/m)(s-t)} \, dW_s \right]^2 \Bigg| \mathcal{F}_{\bar{\mu}_0} \right\}
\mathlarger{\mathlarger{\mathbbm{1}}}_{\bar{\mu}_0 < t} \right\} \\ \nonumber & & +
\sum_{n=1}^{N} \frac{\sigma^2}{2m} \mathbb{E} \left\{ \mathbb{E} \left\{ \left[ \int_{\bar{\mu}_n+S_n'(\omega)}^{t} e^{(\gamma/m)(s-t)} \, dW_s \right]^2 \Bigg| \mathcal{F}_{\bar{\mu}_n + S_n'(\omega)} \right\}
\mathlarger{\mathlarger{\mathbbm{1}}}_{\bar{\mu}_n + S_n'(\omega)<t} \right\} \\ \nonumber
&=& \frac{\sigma^2}{2m} \mathbb{E} \left\{ \int_{\bar{\mu}_0}^{t} e^{2(\gamma/m)(s-t)} \, ds \,
\mathlarger{\mathlarger{\mathbbm{1}}}_{\bar{\mu}_0 < t} \right\} \\ \nonumber & & +
\sum_{n=1}^{N} \frac{\sigma^2}{2m} \mathbb{E} \left\{
\int_{\bar{\mu}_n+S_n'(\omega)}^{t} e^{2(\gamma/m)(s-t)} \, ds \,
\mathlarger{\mathlarger{\mathbbm{1}}}_{\bar{\mu}_n + S_n'(\omega)<t} \right\} \\ \nonumber
&=& \frac{\sigma^2}{4\gamma} \mathbb{E} \left\{ \left[ 1-e^{2(\gamma/m)(\bar{\mu}_0 -t)} \right]
\mathlarger{\mathlarger{\mathbbm{1}}}_{\bar{\mu}_0 < t} \right\} \\ \nonumber & & +
\sum_{n=1}^{N} \frac{\sigma^2}{4 \gamma} \mathbb{E} \left\{
\left[ 1-e^{2(\gamma/m)(\bar{\mu}_n + S_n'(\omega) -t)} \right]
\mathlarger{\mathlarger{\mathbbm{1}}}_{\bar{\mu}_n + S_n'(\omega)<t} \right\} \\ \nonumber
&\le& \frac{\sigma^2}{4 \gamma} \left( N+1 \right) \\ \nonumber
&<& \infty,
\end{eqnarray}
again by the use of the tower property and the It\^o isometry.

For the next step consider $\mathcal{E} \in \mathcal{F}$ such that $\text{Prob}(\mathcal{E})=\delta>0$.
In the case of~\eqref{kt1} we have
\begin{eqnarray}\nonumber
\mathbb{E}\left( \, |K_t| \, \mathlarger{\mathlarger{\mathbbm{1}}}_{\mathcal{E}} \right)
&\le& m e^{-2(\gamma/m) t} \mathbb{E}\left(V_0^2 \, \mathlarger{\mathlarger{\mathbbm{1}}}_{\mathcal{E}} \right)
+ \frac{\sigma^2}{m} \mathbb{E} \left\{ \left[ \int_0^t e^{(\gamma/m)(s-t)} \, dW_s \right]^2
\mathlarger{\mathlarger{\mathbbm{1}}}_{\mathcal{E}} \right\}
\\ \nonumber & & +
\sum_{n=1}^{N} \frac{\sigma^2}{2m} \mathbb{E} \left\{ \left[ \int_{\mu_n+T_n'(\omega)}^{t} e^{(\gamma/m)(s-t)} \, dW_s \right]^2 \,
\mathlarger{\mathlarger{\mathbbm{1}}}_{\mu_n + T_n'(\omega)<t} \, \mathlarger{\mathlarger{\mathbbm{1}}}_{\mathcal{E}} \right\} \\ \nonumber
&\le& m \mathbb{E}\left( V_0^4 \right)^{1/2} \text{Prob}(\mathcal{E})^{1/2}
+ \frac{\sigma^2}{m} \mathbb{E} \left\{ \left[ \int_0^t e^{(\gamma/m)(s-t)} \, dW_s \right]^4
\right\}^{1/2} \text{Prob}(\mathcal{E})^{1/2}
\\ \nonumber & & +
\sum_{n=1}^{N} \frac{\sigma^2}{2m} \mathbb{E} \left\{ \left[ \int_{\mu_n+T_n'(\omega)}^{t} e^{(\gamma/m)(s-t)} \, dW_s \right]^4
\mathlarger{\mathlarger{\mathbbm{1}}}_{\mu_n + T_n'(\omega)<t} \right\}^{1/2} \text{Prob}(\mathcal{E})^{1/2} \\ \nonumber
&=& m \mathbb{E}\left( V_0^4 \right)^{1/2} \delta^{1/2}
+ \frac{\sqrt{3}\sigma^2}{2\gamma} \left[1-e^{-2(\gamma/m)t}\right]
\delta^{1/2}+ \sum_{n=1}^{N} \frac{\sigma^2}{2m} \times
\\ \nonumber & &
\mathbb{E} \left\{ \mathbb{E} \left\{
\left[ \int_{\mu_n+T_n'(\omega)}^{t} e^{(\gamma/m)(s-t)} \, dW_s \right]^4 \,
\Bigg| \mathcal{F}_{\mu_n + T_n'(\omega)} \right\}
\mathlarger{\mathlarger{\mathbbm{1}}}_{\mu_n + T_n'(\omega)<t}
\right\}^{1/2} \delta^{1/2} \\ \nonumber
&\le& m \mathbb{E}\left( V_0^4 \right)^{1/2} \delta^{1/2}
+ \frac{\sqrt{3}\sigma^2}{2\gamma}
\delta^{1/2}+ \sum_{n=1}^{N} \frac{\sqrt{3}\sigma^2}{4\gamma} \times
\\ \nonumber & &
\mathbb{E} \left\{
\left[1-e^{2(\gamma/m)(\mu_n + T_n'(\omega)-t)}\right]^2 \,
\mathlarger{\mathlarger{\mathbbm{1}}}_{\mu_n + T_n'(\omega)<t}
\right\}^{1/2} \delta^{1/2}
\\ \nonumber
&\le& \left[ m \mathbb{E}\left( V_0^4 \right)^{1/2}
+ \frac{\sqrt{3}\sigma^2}{2\gamma}
\left( \frac{N}{2}+1 \right) \right] \delta^{1/2} \\ \nonumber
&<& \varepsilon
\end{eqnarray}
for any $\varepsilon>0$ provided
$$
\delta< \frac{\varepsilon^2}{\left[ m \mathbb{E}\left( V_0^4 \right)^{1/2}
+ \dfrac{\sqrt{3}\sigma^2}{2\gamma}
\left( \dfrac{N}{2}+1 \right) \right]^2},
$$
where we have used~\eqref{young}, H\"older inequality, the tower property, and the normal
distribution of the Wiener integrals
\begin{eqnarray}\nonumber
\int_0^t e^{(\gamma/m)(s-t)} \, dW_s &\sim&
\mathcal{N}\left(0,\frac{m}{2\gamma}\left[1-e^{-2(\gamma/m)t}\right]\right), \\ \nonumber
\int_{\mu_n+T_n'(\omega)}^{t} e^{(\gamma/m)(s-t)} \, dW_s \,
\Bigg| \mathcal{F}_{\mu_n + T_n'(\omega)} &\sim&
\mathcal{N}\left(0,\frac{m}{2\gamma}\left[1-e^{2(\gamma/m)(\mu_n + T_n'(\omega)-t)}\right]\right).
\end{eqnarray}
Consequently there follows uniform integrability, and by the Vitali convergence theorem~\cite{folland} we conclude
\begin{eqnarray}\nonumber
\lim_{t \nearrow \infty} \mathbb{E}\left( |K_t| \right) &=& \mathbb{E}\left( \lim_{t \nearrow \infty} |K_t| \right) \\ \nonumber
&=& 0.
\end{eqnarray}
We employ an analogous argument for~\eqref{kt2}
\begin{eqnarray}\nonumber
\mathbb{E}\left( \, |K_t| \, \mathlarger{\mathlarger{\mathbbm{1}}}_{\mathcal{E}} \right)
&\le& \mathbb{E} \left\{ m e^{-2(\gamma/m) t} V_0^2 \, \mathlarger{\mathlarger{\mathbbm{1}}}_{\mathcal{E}}
+ \frac{\sigma^2}{m} \left[ \int_0^t e^{(\gamma/m)(s-t)} \, dW_s \right]^2 \mathlarger{\mathlarger{\mathbbm{1}}}_{\mathcal{E}}
\right\} \\ \nonumber
&\le& m e^{-2(\gamma/m) t} \mathbb{E}(V_0^4)^{1/2} \text{Prob}(\mathcal{E})^{1/2} \\ \nonumber
&& + \frac{\sigma^2}{m} \mathbb{E} \left\{ \left[ \int_0^t e^{(\gamma/m)(s-t)} \, dW_s \right]^4 \right\}^{1/2} \text{Prob}(\mathcal{E})^{1/2} \\ \nonumber
&\le& m \mathbb{E}(V_0^4)^{1/2} \delta^{1/2}
+ \frac{\sqrt{3}\sigma^2}{2\gamma} \left[1-e^{-2(\gamma/m)t}\right] \delta^{1/2} \\ \nonumber
&\le& \left[ m \mathbb{E}(V_0^4)^{1/2} + \frac{\sqrt{3}\sigma^2}{2\gamma} \right] \delta^{1/2} \\ \nonumber
&<& \varepsilon,
\end{eqnarray}
for any $\varepsilon>0$ provided
$$
\delta < \frac{\varepsilon^2}{\left[ m \mathbb{E}(V_0^4)^{1/2} + \dfrac{\sqrt{3}\sigma^2}{2\gamma} \right]^2}.
$$
By the same convergence theorem we conclude
\begin{equation}\nonumber
\lim_{t \nearrow \infty} \mathbb{E}\left( |K_t| \right) = 0.
\end{equation}
Finally, in the case of~\eqref{kt3}, we have the estimate
\begin{eqnarray}\nonumber
\mathbb{E}\left( \, |K_t| \, \mathlarger{\mathlarger{\mathbbm{1}}}_{\mathcal{E}} \right)
&\le& \frac{\sigma^2}{2m} \mathbb{E} \left\{ \left[ \int_{\bar{\mu}_0}^{t} e^{(\gamma/m)(s-t)} \, dW_s \right]^2 \,
\mathlarger{\mathlarger{\mathbbm{1}}}_{\bar{\mu}_0 < t} \, \mathlarger{\mathlarger{\mathbbm{1}}}_{\mathcal{E}} \right\} \\ \nonumber & & +
\sum_{n=1}^{N} \frac{\sigma^2}{2m} \mathbb{E} \left\{ \left[ \int_{\bar{\mu}_n+S_n'(\omega)}^{t} e^{(\gamma/m)(s-t)} \, dW_s \right]^2 \,
\mathlarger{\mathlarger{\mathbbm{1}}}_{\bar{\mu}_n + S_n'(\omega)<t}
\, \mathlarger{\mathlarger{\mathbbm{1}}}_{\mathcal{E}} \right\} \\ \nonumber
&\le& \frac{\sigma^2}{2m} \mathbb{E} \left\{ \left[ \int_{\bar{\mu}_0}^{t} e^{(\gamma/m)(s-t)} \, dW_s \right]^4 \,
\mathlarger{\mathlarger{\mathbbm{1}}}_{\bar{\mu}_0 < t} \right\}^{1/2} \text{Prob}(\mathcal{E})^{1/2} \\ \nonumber & & +
\sum_{n=1}^{N} \frac{\sigma^2}{2m} \mathbb{E} \left\{ \left[ \int_{\bar{\mu}_n+S_n'(\omega)}^{t} e^{(\gamma/m)(s-t)} \, dW_s \right]^4 \,
\mathlarger{\mathlarger{\mathbbm{1}}}_{\bar{\mu}_n + S_n'(\omega)<t}
\right\}^{1/2} \text{Prob}(\mathcal{E})^{1/2} \\ \nonumber
&=& \frac{\sigma^2}{2m} \mathbb{E} \left\{ \mathbb{E} \left\{ \left[ \int_{\bar{\mu}_0}^{t} e^{(\gamma/m)(s-t)} \, dW_s \right]^4 \Bigg| \mathcal{F}_{\bar{\mu}_0} \right\}
\mathlarger{\mathlarger{\mathbbm{1}}}_{\bar{\mu}_0 < t} \right\}^{1/2} \delta^{1/2}
+ \sum_{n=1}^{N} \frac{\sigma^2}{2m} \times
\\ \nonumber & &
\mathbb{E} \left\{ \mathbb{E} \left\{ \left[ \int_{\bar{\mu}_n+S_n'(\omega)}^{t} e^{(\gamma/m)(s-t)} \, dW_s \right]^4 \Bigg| \mathcal{F}_{\bar{\mu}_n + S_n'(\omega)} \right\}
\mathlarger{\mathlarger{\mathbbm{1}}}_{\bar{\mu}_n + S_n'(\omega)<t} \right\}^{1/2} \delta^{1/2} \\ \nonumber
&=& \frac{\sqrt{3}\sigma^2}{2m} \mathbb{E} \left\{ \left[ \int_{\bar{\mu}_0}^{t} e^{2(\gamma/m)(s-t)} \, ds \right]^2
\mathlarger{\mathlarger{\mathbbm{1}}}_{\bar{\mu}_0 < t} \right\}^{1/2} \delta^{1/2} \\ \nonumber & & +
\sum_{n=1}^{N} \frac{\sqrt{3}\sigma^2}{2m} \mathbb{E} \left\{ \left[
\int_{\bar{\mu}_n+S_n'(\omega)}^{t} e^{2(\gamma/m)(s-t)} \, ds \right]^2
\mathlarger{\mathlarger{\mathbbm{1}}}_{\bar{\mu}_n + S_n'(\omega)<t} \right\}^{1/2} \delta^{1/2} \\ \nonumber
&=& \frac{\sqrt{3}\sigma^2}{4\gamma} \mathbb{E} \left\{ \left[ 1-e^{2(\gamma/m)(\bar{\mu}_0 -t)} \right]^2
\mathlarger{\mathlarger{\mathbbm{1}}}_{\bar{\mu}_0 < t} \right\}^{1/2} \delta^{1/2} \\ \nonumber & & +
\sum_{n=1}^{N} \frac{\sqrt{3}\sigma^2}{4 \gamma} \mathbb{E} \left\{
\left[ 1-e^{2(\gamma/m)(\bar{\mu}_n + S_n'(\omega) -t)} \right]^{2}
\mathlarger{\mathlarger{\mathbbm{1}}}_{\bar{\mu}_n + S_n'(\omega)<t} \right\}^{1/2} \delta^{1/2} \\ \nonumber
&\le& \frac{\sqrt{3}\sigma^2}{4 \gamma} \left( N+1 \right) \delta^{1/2} \\ \nonumber
&<& \varepsilon,
\end{eqnarray}
for any $\varepsilon>0$ whenever
$$
\delta < \frac{16 \gamma^2 \varepsilon^2}{3 \sigma^4 \left( N+1 \right)^2},
$$
where we have employed H\"older inequality, the tower property, and the fact
that the Wiener integrals have the Gaussian distribution
\begin{eqnarray}\nonumber
\int_{\bar{\mu}_0}^t e^{(\gamma/m)(s-t)} \, dW_s \, \Bigg| \mathcal{F}_{\bar{\mu}_0} &\sim&
\mathcal{N}\left(0,\frac{m}{2\gamma}\left[1-e^{2(\gamma/m)(\bar{\mu}_0-t)}\right]\right), \\ \nonumber
\int_{\bar{\mu}_n+S_n'(\omega)}^{t} e^{(\gamma/m)(s-t)} \, dW_s \,
\Bigg| \mathcal{F}_{\bar{\mu}_n + S_n'(\omega)} &\sim&
\mathcal{N}\left(0,\frac{m}{2\gamma}\left[1-e^{2(\gamma/m)(\bar{\mu}_n + S_n'(\omega)-t)}\right]\right).
\end{eqnarray}
Again by the Vitali convergence theorem our conclusion is
\begin{equation}\nonumber
\lim_{t \nearrow \infty} \mathbb{E}\left( |K_t| \right) = 0.
\end{equation}
\end{proof}

Note on one hand that the long time behavior of the physical solutions was computed in section~\ref{kinetic}, where we found
\begin{equation}\nonumber
\lim_{t \nearrow \infty} \mathbb{E}(K_t) = \frac{\sigma^2}{4\gamma}.
\end{equation}
On the other hand, the equipartition theorem of classical statistical mechanics~\cite{pb} imposes
\begin{equation}\nonumber
\lim_{t \nearrow \infty} \mathbb{E}(K_t) = \frac{k_B \, T}{2},
\end{equation}
where $k_B$ is Boltzmann constant and $T$ is the absolute temperature.
The agreement between both is mediated by the fluctuation-dissipation relation~\cite{gillespie}
$$
\sigma^2 = 2 \, k_B \, T \, \gamma.
$$
The asymptotic behavior described by Theorem~\ref{thltb}, which is uniform in the parameter values, is however inconsistent with these.
We thus conclude that the long time behavior of this class of spurious solutions is not physical.

\section{Conclusions}

We can summarize some of our results (precisely, some of the results of section~\ref{direct}) in the following two statements.
The first one is concerned with the solutions to the It\^o equation for the kinetic energy of the Langevin particle.

\begin{theorem}
The stochastic differential equation
\begin{equation}\nonumber
d K_t = \frac{\sigma^2}{2m} dt -2\frac{\gamma}{m} \, K_t \, dt + \sqrt{2 \frac{\sigma^2}{m} \, K_t} \, dW_t,
\qquad \left. K_t \right|_{t=0}= \frac12 m V_0^2,
\end{equation}
admits the unique solution
\begin{equation}\label{us}
K_t = \frac12 m \left[ e^{-(\gamma/m) t} V_0 + \frac{\sigma}{m} \, \int_0^t e^{(\gamma/m)(s-t)} \, dW_s \right]^2.
\end{equation}
\end{theorem}

While the second one refers to the solutions of the Stratonovich equation for the same quantity.

\begin{theorem}
The stochastic differential equation
\begin{equation}\nonumber
d K_t = -2\frac{\gamma}{m} \, K_t \, dt + \sqrt{2 \frac{\sigma^2}{m} \, K_t} \circ dW_t, \qquad \left. K_t \right|_{t=0}= \frac12 m V_0^2,
\end{equation}
admits infinitely many solutions, and the solution set includes~\eqref{us} along with the family
\begin{eqnarray}\nonumber
K_t &=& \frac12 m \left[ e^{-(\gamma/m) t} V_0 + \frac{\sigma}{m} \, \int_0^t e^{(\gamma/m)(s-t)} \, dW_s \right]^2
\, \mathlarger{\mathlarger{\mathbbm{1}}}_{t < T_1(\omega)}
\\ \nonumber & & +
\sum_{n=1}^{\infty} \frac{\sigma^2}{2m} \left[ \int_{\lambda_n+T_n(\omega)}^{t} e^{(\gamma/m)(s-t)} \, dW_s \right]^2 \,
\mathlarger{\mathlarger{\mathbbm{1}}}_{\lambda_n + T_n(\omega)<t<T_{n+1}(\omega)},
\end{eqnarray}
if $V_0(\omega) \neq 0$, for any sequence $\{\lambda_n\}_{n=1}^\infty$ of almost surely non-negative,
$L^0(\Omega)$, and $\mathcal{F}_{T_n(\omega)}-$measurable random variables, where
$$
T_1 := \inf \{ t>0 : K_t=0 \}, \quad T_n := \inf \{ t>T_{n-1} + \lambda_{n-1} + \tau_{n-1} : K_t=0 \}, \quad n=2,3,\cdots,
$$
with $\{\tau_n\}_{n=0}^\infty$ an arbitrary sequence of almost surely positive,
$L^0(\Omega)$, and $\mathcal{F}_{T_{n}(\omega)}-$measurable random variables,
and
\begin{eqnarray}\nonumber
K_t &=& \frac{\sigma^2}{2m} \left[ \int_{\bar{\lambda}_0}^{t} e^{(\gamma/m)(s-t)} \, dW_s \right]^2 \,
\mathlarger{\mathlarger{\mathbbm{1}}}_{\bar{\lambda}_0 < t < S_1(\omega)} \\ \nonumber & & +
\sum_{n=1}^{\infty} \frac{\sigma^2}{2m} \left[ \int_{\bar{\lambda}_n+S_n(\omega)}^{t} e^{(\gamma/m)(s-t)} \, dW_s \right]^2 \,
\mathlarger{\mathlarger{\mathbbm{1}}}_{\bar{\lambda}_n + S_n(\omega)<t<S_{n+1}(\omega)},
\end{eqnarray}
if $V_0(\omega) = 0$,
for any sequence $\{\bar{\lambda}_n\}_{n=0}^\infty$ of non-negative, $L^0(\Omega)$,
and $\mathcal{F}_{S_n(\omega)}-$measurable random variables, and where
\begin{equation}\nonumber
S_n := \inf \{ t>S_{n-1} + \bar{\lambda}_{n-1} + \bar{\tau}_{n-1} : K_t=0 \}, \quad n=1,2,\cdots,
\end{equation}
with $S_0 = \bar{\lambda}_{-1} + \bar{\tau}_{-1}$,
where $\{\bar{\tau}_n\}_{n=0}^\infty$ is an arbitrary sequence of almost surely positive,
$L^0(\Omega)$, and $\mathcal{F}_{S_{n}(\omega)}-$measurable random variables;
also, $\bar{\lambda}_{-1}$ is an arbitrary almost surely non-negative,
$L^0(\Omega)$, and $\mathcal{F}_{0}-$measurable random variable,
and $\bar{\tau}_{-1}$ is an almost surely positive,
$L^0(\Omega)$, and $\mathcal{F}_{\bar{\lambda}_{-1}}-$measurable random variable.
\end{theorem}

More solutions can be found in section~\ref{direct}, among which let us focus on
\begin{equation}\nonumber
K_t = \frac12 m \left[ e^{-(\gamma/m) t} V_0 + \frac{\sigma}{m} \, \int_0^t e^{(\gamma/m)(s-t)} \, dW_s \right]^2
\, \mathlarger{\mathlarger{\mathbbm{1}}}_{t < T_1(\omega)}
\end{equation}
that fulfils both
$$
\lim_{t \to \infty} \mathbb{E}(K_t) = 0
$$
by Theorem~\ref{thltb}, and
$$
\lim_{t \to \infty} K_t = 0 \qquad \text{almost surely}
$$
by Corollary~\ref{zeroas}. The first mode of convergence shows the impossibility for this solution to replicate the results in section~\ref{kinetic};
moreover it is discussed in section~\ref{ltb} how this long time behavior is inconsistent with the equipartition of energy and the
fluctuation-dissipation relation. The second mode of convergence implies convergence in distribution, and therefore the violation of the
Maxwell-Boltzmann distribution of the velocity~\cite{gillespie}, which in turn implies that the kinetic energy should be chi-squared distributed.
It is then clear that this is a spurious rather than a physical solution. The same results were proven for a class of solutions
in section~\ref{ltb}. In general, one can see that all of the solutions, except
the unique solution of the It\^o equation, imply that the Langevin particle is at rest, with null kinetic energy, during some time intervals,
in spite of the presence of non-vanishing thermal fluctuations. Obviously, this is not a physical effect.

The formal stochastic differential equation
$$
\frac{dX_t}{dt}=f(X_t) + g(X_t)\xi_t,
$$
where $\xi_t$ is Gaussian white noise, is a \emph{pre-equation} following van Kampen~\cite{kampen},
and it only becomes an actual equation when a suitable notion of stochastic integral is added.
If this notion is not provided, then at best this pre-equation would admit multiple solutions, at least
one for each possible interpretation of noise. However, the situation for the stochastic differential equation
\begin{equation}\nonumber
d K_t = -2\frac{\gamma}{m} \, K_t \, dt + \sqrt{2 \frac{\sigma^2}{m} \, K_t} \circ dW_t, \qquad \left. K_t \right|_{t=0}= \frac12 m V_0^2,
\end{equation}
is not absolutely different, as it admits infinitely many solutions, only one of which has physical meaning.
Consequently, this is not a valid model to describe the kinetic energy of the Langevin particle,
at least if some further prescription is not added in order to select the physical solution. Such a situation is
not new in finance, where models with multiple solutions have been studied, and the right solution has been
selected by the addition of a new requirement, such as the no-arbitrage assumption~\cite{ander}.
In the present case, considering the Stratonovich stochastic differential equation along with the additional prescription
to ensure the physical character of the unique (by prescription) solution, would be equivalent to directly consider the It\^o equation
\begin{equation}\nonumber
d K_t = \frac{\sigma^2}{2m} dt -2\frac{\gamma}{m} \, K_t \, dt + \sqrt{2 \frac{\sigma^2}{m} \, K_t} \, dW_t,
\qquad \left. K_t \right|_{t=0}= \frac12 m V_0^2.
\end{equation}

In~\cite{kampen} van Kampen studied the direct computation of the kinetic energy of the Langevin particle using the same It\^o and
Stratonovich equations that have been considered herein. The discussion was based on a previous reference~\cite{west}, which claimed the superiority
of the Stratonovich over the It\^o interpretation of noise to compute this quantity. On the other hand, van Kampen claimed the equality
of both approaches, but did not consider the spurious solutions. Herein we have observed a certain advantage of the use of the It\^o
interpretation, as it has not to be supplemented with additional conditions in order to assure the uniqueness of solution.

Overall, van Kampen in~\cite{kampen} concludes that, from a methodological viewpoint, one can use both the It\^o and Stratonovich stochastic differential
equations to model physical systems. From a physical viewpoint, however, he prefers the Stratonovich interpretation whenever the fluctuations are
external. His conclusions were supported 30 years later in~\cite{mmcc}, where the authors claim that the Stratonovich interpretation should be preferred
in the case of a continuous physical system. In this work we have dealt with a continuous physical system influenced by external fluctuations; in fact
a system studied in~\cite{kampen}. We have shown that for this system the Stratonovich interpretation presents an infinite set of spurious solutions
that are not present in the It\^o case. Although this is not a fundamental difficulty, as one can add additional conditions in order to select the
right physical solution in the case of the Stratonovich equation, it is a fact that makes somewhat simpler the It\^o approach. The conclusions
in~\cite{kampen} and~\cite{mmcc} are useful as general guidelines for the modeler, but some of them have to be taken \emph{cum grano salis}.
Physical modeling is crucial in order to select the right interpretation of noise, but the final selection has to be done problemwise;
all in all part of the charm of complex systems is that they rebel against general rules. And just as crucial as physical facts are stochastic
analytical facts. In particular, when one chooses the interpretation of noise in a given problem, one should not disregard neither the validity of
the Watanabe-Yamada theorem for It\^o stochastic differential equations, nor the impossibility to extend it for the Stratonovich ones.

\section*{Acknowledgements}

This work has been partially supported by the Government of Spain (Ministerio de Ciencia, Innovaci\'on y Universidades)
through Project PGC2018-097704-B-I00.

\newpage
\vskip5mm
\noindent
{\footnotesize
Carlos Escudero\par\noindent
Departamento de Matem\'aticas Fundamentales\par\noindent
Universidad Nacional de Educaci\'on a Distancia\par\noindent
{\tt cescudero@mat.uned.es}\par\vskip1mm\noindent
}

\begin{thebibliography}{99}

\bibitem{ander} L. Andersen and J. Andreasen, {\it Volatility skews and extensions of the Libor market model},
Appl. Math. Finance {\bf 7}, 1--32 (2000).

\bibitem{ce} \'{A}. Correales and C. Escudero, {\it It\^o vs Stratonovich in the presence of absorbing states},
J. Math. Phys. {\bf 60}, 123301 (2019).

\bibitem{folland} G. B. Folland, {\it Real Analysis: Modern Techniques and Their Applications}, Wiley, New York, 1999.

\bibitem{fgl} P. Friz, P. Gassiat, and T. Lyons, {\it Physical Brownian motion in a magnetic field as a rough path},
Trans. Amer. Math. Soc. {\bf 367}, 7939--7955 (2015).

\bibitem{gillespie} D. T. Gillespie, {\it The mathematics of Brownian motion and Johnson noise},
Am. J. Phys. {\bf 64}, 225--240 (1996).

\bibitem{ito1} K. It\^o, {\it Stochastic integral}, Proc. Imp. Acad. Tokyo {\bf 20}, 519--524 (1944).

\bibitem{ito2} K. It\^o, {\it On a stochastic integral equation}, Proc. Imp. Acad. Tokyo {\bf 22}, 32--35 (1946).

\bibitem{kuo} H. H. Kuo, {\it Introduction to Stochastic Integration}, Springer, New York, 2006.

\bibitem{langevin} P. Langevin, {\it Sur la th\'eorie du mouvement brownien},
C. R. Acad. Sci. Paris {\bf 146}, 530--533 (1908).

\bibitem{mmcc} R. Mannella and P. V. E. McClintock,
{\it It\^o versus Stratonovich: 30 years later}, Fluctuation Noise Lett. {\bf 11}, 1240010 (2012).

\bibitem{oksendal} B. {\O}ksendal, {\it Stochastic Differential Equations: An Introduction with Applications}, Springer, Berlin, 2003.

\bibitem{nist} F. W. J. Olver, D. W. Lozier, R. F. Boisvert, and C. W. Clark (Eds.),
{\it NIST Handbook of Mathematical Functions}, Cambridge University Press, New York, 2010.

\bibitem{pb} R. K. Pathria and P. D. Beale, {\it Statistical Mechanics}, Elsevier, Oxford, 2011.

\bibitem{stratonovich} R. L. Stratonovich, {\it A new representation for stochastic integrals and
equations}, SIAM J. Control {\bf 4}, 362--371 (1966).

\bibitem{uo} G. E. Uhlenbeck and L. S. Ornstein, {\it On the theory of the Brownian motion},
Phys. Rev. {\bf 36}, 823--841 (1930).

\bibitem{kampen} N. G. van Kampen, {\it It\^o versus Stratonovich}, J. Stat. Phys. {\bf 24}, 175--187 (1981).

\bibitem{wy} S. Watanabe and T. Yamada,
{\it On the uniqueness of solutions of stochastic differential equations}, J. Math. Kyoto Univ. {\bf 11}, 155--167 (1971).

\bibitem{west} B. J. West, A. R. Bulsara, K. Lindenberg, V. Seshadri, and K. E. Shuler,
{\it Stochastic processes with non-additive fluctuations: I. It\^o and Stratonovich calculus and the effects of correlations},
Physica {\bf 97A}, 211--233 (1979).

\end{thebibliography}
\end{document}